\documentclass{amsart}
\usepackage{amsmath,amsfonts,amssymb,amsthm, esint}
\usepackage[english]{babel}
\usepackage{enumerate}
\usepackage{dsfont}
\usepackage{comment}

\usepackage[colorlinks,citecolor=blue]{hyperref}
\usepackage{epsfig,graphicx} 

\newtheorem{theorem}{Theorem}[section]
\newtheorem{definition}[theorem]{Definition}
\newtheorem{lemma}[theorem]{Lemma}

\theoremstyle{remark}

\numberwithin{equation}{section}

\newcommand{\R}{\mathbb{R}}

\title{Minimal points and non-holonomic controllability on compact manifolds}
\author{Sergey Kryzhevich and Eugene Stepanov}

\begin{document}

\begin{abstract}
We study the problem of non-holonomic point-to-point controllability for ODEs with drift possessing some recursion property of the flow (nonwandering or chain recurrence) and satisfying various versions of H\"ormander condition (also known as Lie bracket generating condition). We show that for the flows on compact manifolds, it suffices to assume the validity of the H\"ormander condition on the closure of the set of their minimal points only. Also, we construct a 2-dimensional example of a drift defining a chain recurrent flow and the vector fields defining the non-holonomic constraint, which together satisfy the H\"ormander condition, but the flow is not controllable in the direction of the given vector fields. 
\vspace{10pt}


\vspace{10pt}

\end{abstract}

\maketitle

\emph{This paper is devoted to Prof. Alexander Plakhov's 65 anniversary.}

\section{Introduction}

The Chow-Rashevskii theorem~\cite{Chow39,Rashevskii38} is one of the important results in both control theory and differential geometry. Let a set of smooth vector fields on a smooth connected manifold be such that the Lie algebra generated by those fields spans the tangent bundle of the manifold (this is called H\"ormander or Lie bracket generating condition).

Then one can reach any given point of the manifold from another given point (one says in this case that global controllability holds)
using controls tangent to the planes from the distribution defined by the mentioned vector fields. In other words, if one considers
the distribution of planes defined by the given vector fields as a non-holonomic constraint for the control system, then the 
H\"ormander condition suffices for non-holonomic controllability (here, no drift is supposed). 
  
In the case of ODEs with drift, the analogous controllability problem is more complicated; the known results can be found in~\cite{BScontrol}.
In particular, corollary~31 of~\cite{BScontrol} (see also corollary~26 of the same paper) says that if the drift is nonwandering (called weakly recurrent in the quoted paper), the global controllability with non-holonomic constraint defined by a set of vector fields
is guaranteed if these vector fields together with the drift satisfy the H\"ormander condition. 
Moreover, proposition~27 of~\cite{BScontrol} asserts that when the drift is only chain recurrent but not nonwandering, then global non-holonomic controllability still holds under a stronger condition that the vector fields
defining the constraint to satisfy the H\"ormander condition themselves (i.e., without taking into account the drift).
Remarkably, if the phase space is compact, the control may be taken arbitrarily small. 

In the present paper, we suggest a new approach that uses the properties of the flow. We observe that any pseudotrajectory of a flow passes infinitely many times near the minimal points of the flow (and the time intervals between nearby passages are bounded). Thus, it suffices to assume that the H\"ormander condition is satisfied just near the set of minimal points rather than on the whole phase space. 
Both of the results mentioned from~\cite{BScontrol} are generalized in this sense.
We also show that the second of the mentioned results is rather sharp, by providing an example of a drift generating a chain recurrent (but not nonwandering) flow such that non-holonomic controllability does not take place if the vector fields defining the constraint do not satisfy the H\"ormander condition themselves, even when the H\"ormander condition is valid for these vector fields together with the drift.

The paper is organized as follows. In Section~2, we introduce the principal concepts, definitions, and facts of topological dynamics, mainly related to various types of recurrence. Besides, we discuss the H\"ormander condition and the related properties and statements, including the Chow-Rashevskii theorem. In Section~3, we formulate the existing results on affine controllability. In Section~4, we prove our original results for the case of nonwandering systems. We show that the H\"ormander condition can fail out of the set of minimal points with the main results of Section 3 holding true. Similar results for chain recurrent drifts are given in Section~5. Section~6 contains an example of a non-controllable system with chain recurrent drift showing that one cannot easily 'merge' the above statements, taking the weakest form of recurrence and of the H\"ormander condition. 

\section{Preliminaries}
\subsection{Some results on topological dynamics}

Let $M$ be a $C^1$-smooth compact connected Riemannian manifold, $V$ be a $C^1$-smooth vector field on $M$. Consider the flow $(t,x)\in \R\times M\mapsto \varphi_V(t,x)\in M$ induced by the system
\begin{equation}\label{eq1}
\dot x=V(x),
\end{equation}
that is, $x(\cdot)=\varphi_V(\cdot,x_0)$ is the solution of~\eqref{eq1} with initial conditions $x(0)=x_0$. 

\begin{definition}
We say that a point $x_0\in M$ is minimal with respect to the flow $\varphi_V$, if the closure of the orbit
\[\overline{O_V(x)}=\overline{\{\varphi_V(t,x):t\in {\mathbb R}\}} 
\]
does not contain any proper $\varphi_V$-invariant closed subset. 
\end{definition}

Denote by ${\mathrm{Min}}_V$ the closure of the set of all points minimal with respect to the flow $\varphi_V$. It is well-known that any stationary or periodic point is minimal. The converse statement is, generally speaking, false (one may consider the irrational wrapping of the torus where all the points are minimal but not periodic). Besides, any compact $\varphi_V$-invariant set contains a minimal subset (all the points of that subset are minimal), which follows from the Kuratowski-Zorn lemma. 

\begin{definition}
We say that a point $x_0\in M$ is nonwandering for~\eqref{eq1} (equivalently, for the flow $\varphi_V$), if for any neighborhood $U$ of $x_0$ there exists a couple of points $\{p,q\}\subset U$ 
and a $T\ge 1$ such that $\varphi_V(T,p)=q$. Equivalently, the point $p$ is nonwandering if for every neighborhood $U$ of $p$ there exists a $T\ge 1$
such that $\varphi_V(T,U)\bigcap U\neq \emptyset$.
\end{definition}

Denote the set of all nonwandering points of~\eqref{eq1} by $\Omega_V$. It is known that ${\mathrm{Min}}_{V}\subset \Omega_V$ (see~\cite[\S 3.3]{kaha}). The converse inclusion is, in general, false.

\begin{definition}
	We introduce the notion of a $\delta$-solution and a chain recurrent point of~\eqref{eq1} as follows.
	\begin{itemize}
		\item Given a $\delta>0$, we say that a piecewise smooth function $x\colon [a,b]\mapsto M$ is a $\delta$-solution of~\eqref{eq1}, if for any $t\in [a,b]$ such that the derivative $\dot{t}$ exists, one has
		\[
		|\dot{x} (t)-V(x(t))|_{x(t)}\le \delta,
		\]
		where $|\cdot|_x$ stands for the Riemannian norm in the tangent space $T_xM$. Equivalently, $x(\cdot)$ is a $\delta$-solution of~\eqref{eq1}, if 
		\[
		\dot{x} (t)=V(x(t)) + u(t,x(t)) 
		\]
		for some control function $u$ over $[a,b]\times M$ such that $u(t,\cdot)$ is a vector field over $M$ and 
		\[
		|u(t,x(t))|_{x(t)}\le \delta.
		\]
		\item We say that a point $x_0\in M$ is chain recurrent for~\eqref{eq1} (equivalently, for the flow $\varphi_V$), if for any $\delta>0$ there exists a $T\ge 1$ and a $\delta$-solution $x\colon [0,T)\mapsto M$ of~\eqref{eq1} such that
		\[x(0)=x(T)=x_0.\]
		In fact, we may assume that the value $T$ is greater than any positive constant that does not depend on $\delta$.
		\end{itemize}
\end{definition}

Denote the set of all chain recurrent points by ${\mathrm{CR}}_V$. It is known that
\[
\Omega_V \subset {\mathrm{CR}}_V.
\]
The converse inclusion may be wrong: as an example, one can consider the flow defined by the ODE \[\dot x=1-\cos x\] on the unit circle. 

\begin{definition}
	The ODE~\eqref{eq1} (or, equivalently, the flow $\varphi_V$) is called 
	\begin{itemize}
		\item {\em nonwandering}, if $\Omega_V=M$,
		 \item {\em chain recurrent}, if ${\mathrm{CR}}_V=M$ that is, if for every $p\in M$ and for every $\delta>0$ there is a $T>0$ and
		 a $\delta$-solution $x\colon [0,T]\to M$ of~\eqref{eq1} satisfying $x(0)=x(T)=p$,
		 \item {\em chain transitive}, if for every couple of points $\{p,q\}\subset M$ and for every $\delta>0$ there is a $T>0$ and
		 a $\delta$-solution $x\colon [0,T]\to M$ of~\eqref{eq1} satisfying $x(0)=p$, $x(T)=q$.  
	\end{itemize}
\end{definition}

Clearly, if system~\eqref{eq1} is nonwandering then it is also chain recurrent. Besides, any chain transitive flow is also chain recurrent.

Given a subset $A\subset M$ and a $\sigma>0$, we let $(A)_\sigma$ stand for the $\sigma$-neighborhood of $A$. 
The following statement is valid.

\begin{lemma}\label{lem0} For any $\varepsilon>0$ there is a $\delta>0$ and a $T>0$ such that if $x\colon [0,T]\mapsto M$ is a $\delta$-solution of~\eqref{eq1}, then 
\[\left\{x(t)\colon t\in [0,T]\right\}\bigcap 
(\mathrm{Min}_V)_{\varepsilon}\neq \emptyset.\]
\end{lemma}

\begin{proof} Suppose the statement of the lemma does not hold. Then there exists an $\varepsilon>0$ and a sequence of $\delta_k$-pseudotrajectories $x_k\colon [0,T_k]\mapsto M$ such that 
\begin{enumerate}
\item $\delta_k>0$, $\lim_k \delta_k=0$,
\item $\lim_k T_k\to+\infty$,
\item $\{x_k(t)\colon t\in [0,T_k]\}\bigcap ({\mathrm{Min}}_V)_\varepsilon= \emptyset$.
\end{enumerate}
By the Ascoli-Arcel\'a theorem, there is a subsequence $x_k$ (not relabeled) converging to a solution $x_*$ of~\eqref{eq1} uniformly over the compact intervals of time. Then
\begin{equation}\label{eq_minV1}
	\overline{\left\{x_*(t)\colon t\ge 0 \right\}}\bigcap ({\mathrm{Min}}_V)_\varepsilon= \emptyset.
\end{equation}The $\omega$-limit set for the trajectory of $x_*(t)$ is nonempty in view of the compactness of $M$, hence it contains some minimal points. On the other hand, this set cannot intersect with $({\mathrm{Min}}_V)_\varepsilon$ in view of~\eqref{eq_minV1}, and this contradiction proves the lemma.
\end{proof}

The following more or less folkloric lemma will be used in the proofs below.

\begin{lemma}\label{lm_chaintrans1} Let the flow $\varphi_V$ of the system \eqref{eq1} defined on a connected Riemannian manifold be chain recurrent. Then it is a chain transitive. 
\end{lemma}

\begin{proof}
	Let $M$ be a connected Riemannian manifold equipped with the Riemannian distance $d$, $\{p,q\}\subset M$, $p\neq q$ and $\theta\subset M$ be an arc (i.e.\ an injective curve) with endpoints $p$ and $q$. Assume that $\varepsilon>0$ is given. Divide $\theta$ into $N$ consecutive points $x_i$, with $x_1=p$, and $x_N=q$, so that $d(x_i, x_{i+1})\leq \delta$ for all $i=1,\ldots, N-1$, and $\delta$ to be chosen later. Once we arrive at the point $x_i$ at the instance $t_i\geq 0$, $i=1,\ldots, N-1$, we choose the control $u^1_\varepsilon$  with $|u^1_\varepsilon (\cdot)| < \varepsilon/2$ so as to return to the same point $x_i$ after time $T_i>1$. According to lemma~3.2 from~\cite{KryzhSte21} there is a $\tau \in (0, T_i)$, a $\rho > 0$ both  depending on $\varepsilon$ and a piecewise continuous control $u^2_\varepsilon$  with $|u^2_\varepsilon (\cdot)| < \varepsilon/2$ different from zero only on $(t_i+T_i-\tau, t_i+T_i)$ such that using the control
	\[
	u_\varepsilon:= u^1_\varepsilon+u^2_\varepsilon
	\]
	we arrive at $x_{i+1}$ at the instance $t_{i+1}:= t_i+T_i$, once $\delta$ is chosen so as $\delta\in (0,\rho)$.   
\end{proof}

\subsection{The H\"{o}rmander condition}

Let $M$ be a $C^\infty$ smooth manifold and 
$X_j$, $j=1,\ldots, m$ be vector fields over $M$. We consider the set $\mathcal{X}$ of all $X_j$ and all their continuous Lie brackets (if the vector fields are smooth enough so that the respective brackets will be defined), i.e.\
\[
\mathcal{X}:= \{X_i\}_{i=1}^m \bigcup  \{[X_i, X_j]\}_{i,j=1}^m \bigcup \{[[X_i, X_j], X_k]\}_{i,j,k=1}^m\cup \ldots 
\]

\begin{definition}\label{def_Hor1}
The set of vector fields 
$X_j$, $j=1,\ldots, m$ is said to satisfy the \textit{H\"{o}rmander condition}, if 
\[
\mathrm{span}\, \{Y(x)\colon Y\in \mathcal{X}\} = T_x M.
\]
for every $x\in M$, where $T_x M$ stands for the tangent space to $M$ at $x$.
\end{definition}

The celebrated Chow-Rashevskii theorem asserts that if $M$ is connected (not necessarily compact) and the set of vector fields 
$X_j$, $j=1,\ldots, m$ satisfies the H\"{o}rmander condition, then for every couple of points $\{p,q\}\subset M$ there is a
piecewise smooth curve $x\colon [0,T]\to M$ such that
\begin{align*}
&\dot x(t)=u_1(t)X_1(x(t))+\ldots+u_m(t)X_m(x(t)),\\
&x(0)=p,  \quad x(T)=q
\end{align*}
for some $T>0$ and some piecewise smooth functions $u_j\colon [0,T]\to \R$, $j=1,\ldots, m$.

Moreover, the curve $x(t)$ may be selected so that for any $t$ such that the derivative $\dot x(t)$ exists, one of the functions $u_j$ equals $1$ or $-1$ while all others vanish.

\section{Control affine systems}

In this paper, we aim to study the controllability of the control affine systems of the form
\begin{equation}\label{eq2}
	\dot x(t)=
    V(x(t))+u_1(t)X_1(x(t))+\ldots+u_m(t)X_m(x(t)).
\end{equation}
where $V$, $X_1,\ldots, X_m$ are the given smooth vector fields on $M$. Namely, given a couple of points $\{p,q\}\subset M$, we are interested in finding the piecewise continuous functions $u_j\colon [0,T] \to \R$, $j=1,\ldots, m$ for some $T>0$ such that there is a solution $x\colon [0,T]\to M$ of~\eqref{eq2} satisfying 
\[
x(0)  =p, \quad x(T)=q.
\]
If such functions can be found, then we will call~\eqref{eq2} controllable in $M$. If, moreover, one for every $\delta>0$ one can find such functions $u_j$ that also satisfy
\[
\left| u_1(t)X_1(x(t))+\ldots+u_m(t)X_m(x(t))\right |_{x(t)}\leq \delta,
\]
for all $t\in [0,T$], then~\eqref{eq2} will be called controllable with arbitrarily small controls.

In the paper \cite{BScontrol}, the following two statements are proven (corollaries 31 and 32, respectively).

\begin{lemma} \label{lm_lem1_1}
Let the set of vector fields $\{X_1,\ldots,X_m\}$ satisfy the H\"ormander condition in $M$. Then the system \eqref{eq2} is controllable on $M$.
\end{lemma}

\begin{lemma} \label{lm_lem1_2}
Let the system \eqref{eq1} be nonwandering and the set of vector fields $$\{V, X_1,\ldots,X_m\}$$
satisfies the H\"ormander condition in $M$.
Then the system \eqref{eq2} is controllable on $M$ with arbitrarily small controls.
\end{lemma}

Given two positive numbers $\varepsilon$ and $\tau$, we consider the set ${\mathcal A}_{\varepsilon,\tau,p}\subset M$ of all the points $q$ such that the boundary value problem for some system \eqref{eq2} with all $|u_j(t)|\le \varepsilon$
and boundary conditions $$x(0)=p,\qquad x(\theta)=q$$
is solvable for a $\theta\in (0,\tau]$.
We also consider the sets
\[
{\mathcal A}_{\varepsilon,p} := \bigcup_{\tau>0} {\mathcal A}_{\varepsilon,\tau,p}, \quad 
{\mathcal A}_{p} := \bigcup_{\varepsilon>0} {\mathcal A}_{\varepsilon,p},
\]
that is the sets of all the points attainable from $p$ using controls bounded in the uniform norm by $\varepsilon$, and using controls with arbitrary norm, respectively.

We quote a classical Control Theory statement also given in~\cite{BScontrol}.

\begin{lemma} \label{lm_lem1_3}(Krener's theorem).
Let the set of vector fields $$\{V, X_1,\ldots,X_m\}$$
satisfy the H\"ormander condition on $M$. Then for any $\tau>0$ and any $\varepsilon>0$, the point $p$ belongs to the closure of the interior of the set ${\mathcal A}_{\varepsilon,\tau,p}$. In particular, the latter interior is non-empty.
\end{lemma}

In the next two sections, we generalize the results of Lemmata~\ref{lm_lem1_1} and~\ref{lm_lem1_2}. In particular, we show that it suffices to assume that the H\"ormander condition is satisfied on a neighborhood of the set of minimal points. 

\section{The nonwandering drift}

We show now that for controllability of systems with nonwandering drift, it is enough that the  H\"ormander condition be satisfied only over the closure of the set of all minimal points.

\begin{theorem}\label{th_Chow1}
	Let the system \eqref{eq1} be nonwandering and suppose that 
 the set of vector fields $\{V, X_1,\ldots,X_m\}$ satisfies the H\"ormander condition on ${\mathrm{Min}}_V$.
	Then the system \eqref{eq2} is controllable with arbitrarily small controls.
\end{theorem}

\begin{proof} 
	Under the conditions of the theorem being proven
	we have that 
		for any point $x_0\in {\mathrm{Min}}_V$ there is a neighborhood $U(x_0)$ such that the set 
		$\{V, X_1,\ldots,X_m\}$ satisfies the H\"ormander condition in $U(x_0)$.
We take a neighborhood $U_M$ of the set ${\mathrm{Min}}_V$ such that the considered set of vector fields satisfies the H\"ormander condition on $U_M$.
	We fix an $\varepsilon>0$ and later on, in this section, we only consider systems~\eqref{eq2} with controls $|u_j|\le \varepsilon$.

We assume $M$ to be equipped with its intrinsic (geodesic) Riemannian distance $d$. Since ${\mathrm{Min}}_V$ is compact, we can find a $\sigma>0$ such that for the ball $B_\sigma (x_0)$ of radius $\sigma$ centered at $x_0$ one has $B_\sigma (x_0)\subset U_M$ for every $x_0\in {\mathrm{Min}}_V$. Given this $\sigma$, we take $\delta$ and $T$ according to Lemma~\ref{lem0} with $\sigma/2$ in place of $\varepsilon$, i.e., so that for every $\delta$-solution $\psi(\cdot)$ of~\eqref{eq1} to have
	\begin{equation}\label{eq_pstMinV1}
		\left\{\psi(t)\colon t\in [0,T]\right\}\bigcap 
		(\mathrm{Min}_V)_{\sigma/2}\neq \emptyset.
	\end{equation} 
	Of course, without loss of generality, we may increase $T$ if necessary to have $T>2$.
We may now further decrease  $\delta>0$  taking it so small that additionally for every $\delta$-solution $\psi(\cdot)$ of~\eqref{eq1} and every $t\in [-4T,4T]$ we have
\begin{equation}\label{eq_pstMinV2}
	d(\psi(t),\varphi_V(t,\psi(0))\le \sigma/2.
\end{equation} 	

Consider now an arbitrary couple of points $\{p,q\}\subset M$.  The chain recurrence implies chain transitivity by Lemma~\ref{lm_chaintrans1}, and thus there exists a $\delta$-solution 
$\psi\colon [0,\tau]\mapsto M$ of~\eqref{eq1} such that $\psi(0)=p$, $\psi(\tau)=q$. Since $\tau$ may be taken as large as we like, we assume $\tau >T$ and set
	\[
	N:=\left\lfloor \frac{\tau}{T}\right\rfloor,
	\] 
	so that $N\geq 1$ and $NT \leq \tau \leq (N+1)T$ and hence
	\[
	T\leq \frac{\tau}{N}\leq 2T.
	\]
	Thus, in view of~\eqref{eq_pstMinV1} we can find points $y_j\in \mathrm{Min}_V$
	and instants of time $t_j \in [(j-1)\tau/N, j\tau/N]$
	such that 
\begin{equation}\label{eq_psrMinV4}
		d(\psi(t_j),y_j)\le \sigma/2 
\end{equation}	
		for all $j=1,\ldots,N$.

Now we construct the control using the following algorithm (see Fig.~\ref{fig_global}). 
\begin{itemize}
\item We start at the point $p$ and let the control functions $u_k(t):=0$ for $t\in [0, t_1]$, $k-1,\ldots, m$, that is, we just follow the flow induced by the drift $V$ until the instance $t_1$ thus arriving at a point $x_1:=\varphi_V(t_1, p)$. Note that \[0\leq t_1\leq \tau/N\leq 2T,\] and therefore $d(x_1, \psi(t_1))\le \sigma/2$ by~\eqref{eq_pstMinV2} so that by~\eqref{eq_psrMinV4}  we get \[d(x_1, y_1)\le \sigma.\]
Thus both $x_1\in B_\sigma(y_1)$ and $\psi(t_1)\in B_\sigma(y_1)$, and therefore by Lemma \ref{lm_local} (given below) we may use the control functions $u_1,\ldots, u_m$ to get from $x_1$ to $\psi(t_1)$. 
\item Similarly, for every $j=1,\ldots, N-1$ once we arrive a point $\psi(t_j)$ at some instance $t_j'$ we let the control functions
\[u_1(t)=\ldots=u_m(t):=0 \quad\mbox{for }t\in [t_j', t_j'+(t_{j+1}-t_j)],\] that is, we just follow the flow induced by the drift $V$ for $t_{j+1}-t_j$ instances of time, thus arriving at a point 
\[	x_{j+1}:=\varphi_V(t_{j+1}-t_j, \psi(t_j)).\]	
One clearly has
\[
t_{j+1}-t_j \leq 2\frac{\tau}{N}\leq 4T,
\]
so that $d(x_{j+1}, \psi(t_{h+1}))\le \sigma/2$ by~\eqref{eq_pstMinV2}, hence by~\eqref{eq_psrMinV4}  
one has \[d(x_{j+1}, \psi(t_{j+1}))\le \sigma.\]
Thus both $x_{j+1}\in B_\sigma(y_{j+1})$ and $\psi(t_{j+1}) \in B_\sigma(y_{j+1})$, and therefore we may use the control functions $u_1,\ldots, u_m$ to get from $x_{j+1}$ to $\psi(t_{j+1})$.
\item  Finally, once we arrive a point $\psi(t_N)$ at some instance $t_N'$, setting 
\[x_*:= \varphi_V (-(\tau-t_N), q),
\]
we get $d(x_*, \psi(t_N))\leq \sigma/2$ by~\eqref{eq_pstMinV2} since 
\[
0\leq \tau-t_N \leq \frac{\tau}{N}\le 2T.
\]
Thus by~\eqref{eq_psrMinV4} one has \[d(x_*, y_N)\le \sigma,\] which implies that both $x_*\in B_\sigma(y_N)$ and $\psi(t_N) \in B_\sigma(y_N)$, and therefore we may use the control functions $u_1,\ldots, u_m$ to get from $\psi(t_N)$ to $x_*$. It remains then to set $u_1,\ldots, u_m$ to zero, i.e.\ follow the flow of $V$ for $\tau-t_N$ instances of time to arrive finally at $q$.
\end{itemize}

\begin{figure}[!ht]
\centering
\includegraphics[height=2in]{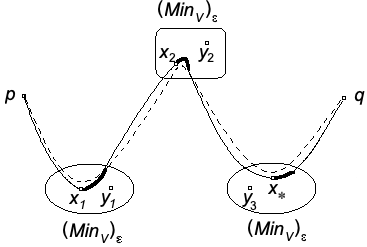}
\caption{Control in neighborhoods of the minimal sets. The dashed line stands for $\psi(t)$, the regular lines for exact trajectories, and the bold lines for trajectories of systems with a control} \label{fig_global}
\end{figure}

Note that the control functions we constructed are all zero unless we had to connect points in small neighborhoods of some minimal point of~\eqref{eq1}, so that if inside these neighborhoods they can be taken arbitrarily small, then our construction leads also to arbitrarily small controls.  
\end{proof}

The following statement on local controllability was used in the above proof.

\begin{lemma}\label{lm_local}
	Let the flow $\varphi_V$  be nonwandering, and the point $x_0$ be such that the vector fields 
	$$\{V,X_1,\ldots,X_m\}$$
	satisfy the H\"ormander condition in a neighborhood $U(x_0)$ of $x_0$. Then the point $x_0$ belongs to the interior of the set ${\mathcal A}_{\varepsilon,x_0}$.
\end{lemma}

\begin{proof} 
	We divide the proof into several steps.
	
	\textbf{Step 1.} Consider the set of vector fields 
	\[\begin{array}{l} 
		{\mathcal Y} 
		:=
		\{V, -V\}\cup \{\pm V\pm \varepsilon X_k\}_{k=1}^m.
	\end{array}
	\]
The set ${\mathcal Y}$ still satisfies the H\"ormander condition.
Thus by the classical Chow-Rashevskii theorem (theorem~17 of \cite{BScontrol}) every $x_0\in \R^d$ admits a neighborhood $U^0(x_0)$ with any two points of $U^0(x_0)$ being linked by a piecewise smooth curve, such that tangent lines to that curve (where they exist) are parallel to one of the vector fields from ${\mathcal Y}$. Up to restricting $U(x_0)$ we may suppose without loss of generality 
$U(x_0)=U^0(x_0)$.

\textbf{Step 2.} Now we claim that for any point $x_1\in U(x_0)\bigcap {\mathcal A}_{\varepsilon,x_0}$ and any $\theta>0$ the points 
$\varphi_V(-\theta,x_1)$ and  $\varphi_{V\pm \varepsilon X_i}(-\theta,x_1)$
belong to the closure of the set ${\mathcal A}_{\varepsilon, x_1}$ and, hence, to the closure of ${\mathcal A}_{\varepsilon,x_0}$.
	
\textbf{Step 2.1.} We prove the claim for $\varphi_V(-\theta,x_1)$.
Indeed, by Krener's theorem (Lemma~\ref{lm_lem1_3}), the interior of the set ${\mathcal A}_{\varepsilon, x_1}$ is non-empty and intersects with the open ball $B_{\delta}(x_1)$ for any $\delta>0$. Since the flow $\varphi_V$ is nonwandering, for any $\delta>0$ there exists a $T_1>\theta$ and two points $x_2\in {\mathcal A}_{\varepsilon,x_1}\cap B_\delta(x_1)$ and $x_3=\varphi_V(T_1,x_2)\cap B_\delta(x_1)$
(hence also $x_3\in {\mathcal A}_{\varepsilon,x_1}\cap B_\delta(x_1)$). 
	
We observe now that for any $\theta>0$ the point $x_{3,\theta}:=\varphi_V (T_1-\theta,x_3)$ belongs to the set ${\mathcal A}_{\varepsilon,x_1}$. On the other hand, $x_{3,\theta}=\varphi_V(-\theta,x_3)$ and since $x_3\in B_\delta(x_1)$, we get
$x_{3,\theta}\in \varphi_V(-\theta, B_{\delta}(x_1))$. Choosing $\delta$ sufficiently small we thus get the point $x_{3,\theta}$ as close to $\varphi_V(-\theta,x_1)$ as we like, which proves the claim.
	
\textbf{Step 2.2.} We now approximate the points $\varphi_{V\pm \varepsilon X_j}(-\theta,x_1)$ (we will do this for
$\varphi_{V- \varepsilon X_j}(-\theta,x_1)$, the argument for $\varphi_{V+ \varepsilon X_j}(-\theta,x_1)$ being completely symmetric).
To this aim,  we suggest a construction resembling Euler's numerical method. First, we do it for small values of $\theta$ and then iterate.
First of all, we observe that there exists a $C>0$ depending only on $V$, $X_j$ and $\varepsilon \in \R$, such that the distance between points 
\[
x_{4,\theta}:=\varphi_{V+\varepsilon X_j} (\theta, \varphi_V(-2\theta,x_1)) \quad \mbox{and }x_{5,\theta}:=\varphi_{V-\varepsilon X_j} (-\theta, x_1)
\]
does not exceed $C\theta^2/2$. To prove this, it suffices to consider the Taylor decompositions for $x_{4,\theta}$ and $x_{5,\theta}$ as functions of $\theta$, observing that the zero and first order in $\theta$ terms of these decompositions coincide.

Then, similarly to Step~2.1,   we take the points $x_2\in {\mathcal A}_{\varepsilon,x_1}\cap B_\delta(x_1)$ and $x_3=\varphi_V(T_1,x_2)\cap B_\delta(x_1)$
(hence also $x_3\in {\mathcal A}_{\varepsilon,x_1}\cap B_\delta(x_1)$)
 for some $T_1>2\theta$. 
We may take $\delta$ so small that the distance between $x_{4,\theta}$ and the point
\[x_{6,\theta}:=\varphi_{V+\varepsilon X_j} (\theta, \varphi_V(-2\theta,x_3))
\] be not greater than $C\theta^2/2$.
Observe that $x_{6,\theta}\in {\mathcal A}_{\varepsilon,x_1}$ and the distance between 
$x_{5,\theta}$ and $x_{6,\theta}$ is not greater than $C\theta^2$.
	
Now take any value $\theta>0$ and approximate the point $x_{5,\theta}=\varphi_{V- \varepsilon X_j}(-\theta,x_1)$ as follows. 
For an $N\in {\mathbb N}$ set $\xi_0:=x_1$. Then we recursively define the points $\xi_i$
as follows: every $\xi_i$ belongs to ${\mathcal A}_{\varepsilon,\xi_{i-1}}$ (and, consequently, to ${\mathcal A}_{\varepsilon,x_1}$) and is such that
the distance between the points $\xi_i$ and $\varphi_{V-\varepsilon X_j} (-\theta/N, \xi_{i-1})$ is not greater than $C\theta^2/N^2$. In other words, one may take $\xi_1:= x_{6,\theta/N}$, and
every $\xi_i$ is constructed from $\xi_{i-1}$ in the same way as $x_{6,\theta/N}$ is constructed from $x_1$.
Let $d(\cdot,\cdot)$ be the distance on the manifold $M$. Given the value $\theta>0$, we take the constant $L>0$ so that for any pair of points $z_0,z_1 \in M$, any 
$j=1,\ldots,m$ and any $\tau\in [0,\theta]$ we have 
\[
d(\varphi_{V-\varepsilon V_j}(-\tau,z_0),\varphi_{V-\varepsilon V_j}(-\tau,z_1))\le Ld(z_0,z_1),
\]
where for the sake of brevity we denote $\varphi:=\varphi_{V-\varepsilon V_j}$.
The latter formula implies 

\[
\begin{aligned}
		d & \left(\varphi\left(-\dfrac{(i-1)\theta}{N},\xi_{N-i+1}\right),\varphi\left(-\dfrac{i\theta}{N},\xi_{N-i}\right)\right)=\\
		&\quad d\left(\varphi\left(-\dfrac{(i-1)\theta}{N},\xi_{N-i+1}\right),\varphi\left(-\dfrac{(i-1)\theta}{N},\varphi\left(-\dfrac{\theta}{N},\xi_{N-i}\right)\right)\right)\le \\
		&\quad L d\left(\xi_{N-i+1},\varphi\left(-\dfrac{\theta}{N},\xi_{N-i}\right)\right)\le \dfrac {LC\theta^2}{N^2}.
\end{aligned}
\]

Consequently,
\begin{align*}
d& (\xi_N,\varphi_{V-\varepsilon X_j}(-\theta,x_1))=
d(\xi_N,\varphi_{V-\varepsilon X_j}(-\theta,\xi_0))\le\\
&\sum\limits_{i=1}^N d\left(\varphi_{V-\varepsilon V_j}\left(-\dfrac{(i-1)\theta}{N},\xi_{N-i+1}\right),\varphi_{V-\varepsilon V_j}\left(-\dfrac{i\theta}{N},\xi_{N-i}\right)\right)\le \dfrac{LC\theta^2}{N}
\end{align*}
which can be taken arbitrarily small provided $N$ is large.
	
\textbf{Step 3.} By Steps 1 and 2, there exists a $\rho>0$ such that the set ${\mathcal A}_{\varepsilon,x_0}$ is dense in $B_\rho (x_0)$. We prove that in fact every point $y_0\in B_\rho (x_0)$ belongs to the set ${\mathcal A}_{\varepsilon,x_1}$, which would conclude the proof of the lemma. 
	
Consider the reversed-time problem 
\begin{equation}\label{eq_time_rev}
		\dot x=-V(x)+\sum_{j=1}^m u_j(t)X_m(x)
\end{equation}
with initial conditions $x(0)=y_0$ and the same assumptions on the control functions $u_j(t)$. The H\"ormander condition is evidently satisfied for the vector fields 
$\{-V,X_1,\ldots,X_m\}$. Given $y_0$ and $\varepsilon$, let $\widetilde{\mathcal A}_{\varepsilon,y_0}$ be the analogue of the set ${\mathcal A}_{\varepsilon,y_0}$ for system \eqref{eq_time_rev}.
	
By the Krener's theorem (Lemma~\ref{lm_lem1_3}), applied to~\eqref{eq_time_rev}, the closure of the interior of the set $\widetilde{\mathcal A}_{\varepsilon,y_0}$ contains the point $y_0$, hence in particular the latter interior is non-empty.
But due to the density of ${\mathcal A}_{\varepsilon,x_0}$ this interior 
of $\widetilde{\mathcal A}_{\varepsilon,y_0}$
contains a point $z_0\in B_\rho(x_0)\cap {\mathcal A}_{\varepsilon,x_0}$. Evidently, $y_0\in {\mathcal A}_{\varepsilon,z_0}\subset {\mathcal A}_{\varepsilon,x_0}$, which shows the claim and hence finishes the proof of the lemma. 
\end{proof}

\section{The chain recurrent drift}

Now, we only assume that the flow $\varphi_V$ is chain recurrent and show that, for the controllability of the respective system, it suffices to have the validity of the restricted H"ormander condition (only in vector fields $\{X_1,\ldots, X_m\}$) on closing the set of minimal points.
 Here, however, we do not prove that the control functions $u_j$ can be chosen small. Moreover, in our construction, they are large compared to the drift.

\begin{theorem} \label{th_controlMn2}
Let the system \eqref{eq1} be chain recurrent and the set of vector fields $\{X_1,\ldots,X_m\}$ satisfy the H\"ormander condition on ${\mathrm{Min}}_V$.
Then~\eqref{eq2} is controllable in the whole manifold $M$.
\end{theorem}

\begin{proof}
The proof repeats that of Theorem~\ref{th_Chow1} only using Lemma~\ref{lm_local1} instead of Lemma~\ref{lm_local} (observe that in the proof of Theorem~\ref{th_Chow1}, the nonwandering property of the drift was only used in Lemma~\ref{lm_local}, for the rest one needs only chain recurrence).
\end{proof}

The following local statement (in fact, very similar to proposition~27 of \cite{BScontrol}) has been used in the above proof. 

\begin{lemma}\label{lm_local1}
	Let the flow $\varphi_V$ be chain recurrent, and the point $x_0$ be such that the vector fields 
	$$\{X_1,\ldots,X_m\}$$
	satisfy the H\"ormander condition in a neighborhood $U(x_0)$ of $x_0$. Then the point $x_0$ belongs to the interior of the set 
	${\mathcal A}_{x_0}$
	of all the points attainable from $x_0$. 
\end{lemma}

\begin{proof} First of all, we observe that due to the Chow-Rashevskii theorem (theorem~17 of \cite{BScontrol}), applied to the vector fields 
	$\{X_1,\ldots,X_m\}$, there exist $\varepsilon,\delta>0$ and a value $N\in {\mathbb N}$ such that for any $y_0\in B_{2\delta} (x_0)$, there exist values 
	$$\varepsilon_1,\ldots, \varepsilon_N \in 
	(-\varepsilon,\varepsilon)$$ such that 
	\begin{equation}\label{eqe1en}
		y_0=\varphi_{Y_N}(\varepsilon_N, \varphi_{Y_{N-1}}(\varepsilon_{N-1}, \ldots, \varphi_{Y_1}(\varepsilon_1,x_0)\ldots )).
	\end{equation}
	
	Here the vector fields $Y_i$ are defined as follows: $Y_i=X_{(i \mod m) +1}$. Denote the right-hand side of~\eqref{eqe1en} by $\Phi(\varepsilon_1,\ldots,\varepsilon_N)$. This is a continuous function of its arguments. Given a positive number $\sigma \ge 0$, we set 
	$$\Phi_{\sigma} (\varepsilon_1,\ldots,\varepsilon_N)=
	\varphi_{\sigma V+Y_N}(\varepsilon_N, \varphi_{\sigma V+Y_{N-1}}(\varepsilon_{N-1}, \ldots, \varphi_{\sigma V+Y_1}(\varepsilon_1,x_0)\ldots )).
	$$
	
	There exists a $\sigma_0 >0$ such that
	$$\Phi_{\sigma_0}([-\varepsilon,\varepsilon]^N)\supset B_\delta (x_0).$$
In other words, any point of $B_\delta (x_0)$ can be attained from $x_0$ by shifts along the trajectories of the vector fields $\sigma V\pm X_i$ (or of the rescaled vector fields $V\pm \sigma^{-1} X_i$), which concludes the proof.
\end{proof}

\section{A non-controllable system with a chain recurrent drift}

It is quite natural to ask whether the system with a chain-recurrent drift $V$ and the set of vector fields $\{V,X_1,\ldots,X_m\}$ (including the drift) satisfying the H\"ornmander condition is controllable. The counterexample given in the next section shows that the answer could be negative even if the H\"ornmander condition is fulfilled everywhere.  

Let $M$ be the two-dimensional flat torus defined as usual as the square 
$[0,2\pi]\times [0,2\pi]$ with opposite sides identified.
Consider the system of ordinary differential equations on $M$
\begin{equation}\label{eq3}
\left\{
\begin{array}{l}
\dot x=0,\\
\dot y=(1-\cos x) \sin y
\end{array}
\right.
\end{equation}
Let $z:=(x,y)^{T}$; $V(z):=\,(0,(1-\cos x) \sin y)^T$ be the vector field defining the drift of the system~\eqref{eq3}, and $\varphi_V$ be the respective flow (see Fig.\,\ref{przyklad}).

\begin{figure}[!ht]
\centering
\includegraphics[height=3in]{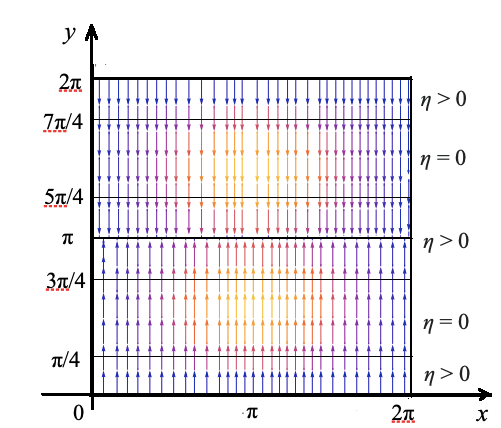}
\caption{The drift vector field of~\eqref{eq3} and the function $\eta$} \label{przyklad}
\end{figure}

All points $(0,y)$, $y\in [0,2\pi]$, $(x,0)$, $(x,\pi)$, $x\in [0,2\pi]$ are stationary, and the system~\eqref{eq3} has no other nonwandering points. Nevertheless, the following statement is true.

\begin{lemma} The flow induced by the system \eqref{eq3} is chain transitive (and, hence, chain recurrent).
\end{lemma}

\begin{proof} We show how a point $(x_1,y_1)$ may be attained from a point $(x_0,y_0)$ with controls with arbitrarily small uniform norm.

{\em Step 1.} The set of stationary points of~\eqref{eq3} is the union of circles 
\begin{align*}
 S &:= S_1\cup S_2\cup S_3,\quad\mbox{where}\\
  &  S_1:= \{(0,y)\colon y\in [0,2\pi]\}, S_2:= \{(x,0)\colon  x\in [0,2\pi]\}, S_3:=  \{ (x,\pi)\colon x\in [0,2\pi]\}.
\end{align*}
Every point of $S_1$ can be reached from any other point of $S_1$ in finite time by applying an arbitrarily small control along the vertical direction.  Similarly, every point of $S_2$ can be reached from any other point of $S_2$ in finite time by applying an arbitrarily small control along the horizontal direction, and the same is true for $S_3$. Therefore, since $S$ is connected, every point of $S$  can be reached from any other point of $S$ in finite time by applying an arbitrarily small control.

{\em Step 2.} If a point $(x_0,y_0)$ is not stationary, then its set of limits $\omega$ is a singleton $\{(x_0,\pi)\}$. We can reach any neighborhood of the latter point from $(x_0, y_0)$ without any control (that is, just following the flow of $V$) and then apply a small control along the vertical direction
to reach the point $(x_0,\pi)\in S$ in finite time.
  
{\em Step 3.} A non-stationary point $(x_1,y_1)$ with $y_1\in (\pi, 2\pi)$ can be reached in finite time from the point $(2\pi,y_1)\in S$ by a small control along the vertical direction and then following the flow of $V$. Analogously, if $y_1\in (0,\pi)$, then a nonstationary point $(x_1,y_1)$ can be reached in finite time from the point $(0,y_1)\in S$ by a small control along the vertical direction and then following the flow of $V$. 

{\em Step 4.} Combining Steps~1, 2, and~3 we get that every point of $M$ can be reached in finite time from any other point of $M$ by arbitrarily small controls, which proves the lemma.
\end{proof}

Now, we consider a function $\eta \in C^\infty (S^1)$ such that \[\eta (x)>0 \quad \mbox{for } 
x\in [0,\pi/4)\bigcup (3\pi/4,5\pi/4) \bigcup (7\pi/4, 2\pi]\] (see Fig. \ref{przyklad}). Introduce the vector fields on the torus by the following formulas:
\[X_1=\begin{pmatrix}1\\ 0\end{pmatrix}, \qquad X_2=\begin{pmatrix}0\\ \eta(y)\end{pmatrix}.
\]    

\begin{lemma} The triple $\{V,X_1,X_2\}$ satisfies the H\"ormander condition on all the torus $M$.
\end{lemma}

\begin{proof} For any $(x,y)\in M$, such that $x\neq 0$, $y\neq 0$, $y\neq \pi$, the vectors $V(x,y)$ and $X_1(x,y)$ form the basis of the tangent space $T_{(x,y)}M$. The vectors $X_1(x,y)$ and $X_2(x,y)$ span the tangent space at the points $(x,0)$ and $(x,\pi)$ for any $x\in [0,2\pi)$.

Let us calculate now the Lie bracket
$$[V,X_1](x,y)=DX_1(x,y) V (x,y)=\dfrac{\partial V}{\partial x}(x,y)=
\begin{pmatrix}0\\ \sin x\sin y\end{pmatrix}.$$
This vector field vanishes at $x=0$. However, the consequent Lie bracket
$$[[V,X_1],X_1](x,y)=\dfrac{\partial }{\partial x}[V,X_1](x,y)=
\begin{pmatrix}0\\ \cos x\sin y\end{pmatrix}$$
is nonzero for $x=0$ and $y\notin\{0,\pi\}$. This proves the lemma. 
\end{proof}

Nevertheless, the following statement is true.

\begin{theorem} The system $\dot z=V(z)+u_1(t)X_1(z)+u_2(t)X_2(z)$, where $z:=(x,y)^T$ is not controllable.
\end{theorem}

\begin{proof} 
Suppose that there exist piecewise continuous functions 
$u_1, u_2\colon [0,T]\mapsto {\mathbb R}$ such that the problem
$$\left\{
\begin{array}{l}
\dot z(t) = V(z(t))+u_1(t)X_1(z(t))+u_2(t)X_2(z(t)),\\
z(0)=(0,\pi), \qquad z(T)=(0,0)
\end{array}
\right.$$
is solvable for some $T>0$. Consider the solution $z(\cdot)= (x(\cdot), y(\cdot))^T$. Since $y(0)=\pi$, $y(T)=0$, at least one of the following statements is true:
\begin{enumerate}
    \item there is $t_0\in (0,T)$ such that $y(t_0)\in (\pi/4,3\pi/4)$, $\dot y(t_0)<0$ or
    \item there is a $t_0\in (0,T)$ such that $y(t_0)\in (5\pi/4,7\pi/4)$, $\dot y(t_0)>0$.
\end{enumerate}
Without loss of generality, we assume that the first one is true (if the second inclusion takes place, the proof is similar).
Then $\eta(y(t_0))=0$ and, consequently, $X_2(z(t_0))=0$.
Since the second component of the vector field $X_1$ is zero, we have
$$\dot y(t_0)=(1-\cos x(t_0))\sin y(t_0)\ge 0$$
which contradicts our assumption. This proves the lemma. \end{proof}

\section*{Acknowledgments} Sergey Kryzhevich was supported by Gda\'{n}sk University of Technology
by the DEC 14/2021/IDUB/I.1 grant under the Nobelium - ‘Excellence Initiative - Research University’ program.
Eugene Stepanov acknowledges the MIUR Excellence Department Project awarded to the Department of Mathematics,
University of Pisa, CUP I57G22000700001. His work is also partially within the framework of HSE University Basic Research Program and of the Ministry of Science and Higher Education of the Russian Federation (agreement 075-15-2025-344 for Saint Petersburg Leonard Euler International Mathematical Institute at PDMI RAS).


\begin{thebibliography}{99} 
\bibitem{BScontrol} Boscain, U., Sigalotti, M., \textit{Introduction to controllability
of nonlinear systems}, in: Contemporary Research in Elliptic PDEs and Related Topics, Ed.\ by Dipierro S.,
203--219, Springer INdAM Ser., vol.\ 33, Springer, Cham, 2019.
\bibitem{Chow39} Chow, W.-L., \"{U}ber Systeme von linearen partiellen Differentialgleichungen erster Ordnung, \textit{Mathematische Annalen} 117 (1940), 98--105.
\bibitem{kaha} A.\, Katok, B.\, Hasselblatt, \emph{Introduction to the Modern Theory of Dynamical Systems}, Cambridge University Press, 1997.
\bibitem{KryzhSte21} Kryzhevich, S., Stepanov, E., Constructive controllability for incompressible vector fields,
\textit{Tchemisova, T.V., Torres, D.F.M., Plakhov, A.Y. (eds) Dynamic Control and Optimization. DCO 2021. Springer Proceedings in Mathematics \& Statistics)} \textbf{407} (2021),  3--18.
\bibitem{Rashevskii38} Rashevski\v{i}, P.K., About connecting two points of complete nonholonomic space
by admissible curve (in Russian), \textit{Uch.\ Zapiski Ped.\ Inst.\ im.\ Liebknechta Ser.\ Phys.\ Math.} 2 (1938),
83--94.
\end{thebibliography}
\end{document}